\documentclass[11pt, a4paper]{amsart}

\usepackage{amsmath, amsthm, amssymb,fullpage}

\usepackage{color}
\usepackage{xcolor}

\newcommand{\bD}{\mathbb{D}}
\newcommand{\bC}{\mathbb{C}}
\newcommand{\bN}{\mathbb{N}}
\newcommand{\bP}{\mathbb{P}}
\newcommand{\bR}{\mathbb{R}}

\theoremstyle{plain}
\newtheorem{theorem}{Theorem}[section]

\newtheorem{lemma}[theorem]{Lemma}
\newtheorem{proposition}[theorem]{Proposition}
\newtheorem*{claim}{Claim}  
\newtheorem{mainth}{Theorem}
\newtheorem*{maincoro}{Corollary}
\theoremstyle{definition}

\newtheorem*{acknowledgement}{Acknowledgement}
\theoremstyle{remark}

\numberwithin{equation}{section}

\renewcommand{\P}{\mathbb{P}}
\newcommand{\C}{\mathbb{C}}

\newcommand{\cC}{\mathcal{C}}

\begin{document} 
\title[Degeneration to a H\'enon map]{Degeneration of quadratic polynomial endomorphisms to a H\'enon map}

\begin{author}[F.~Bianchi]{Fabrizio Bianchi}
\address{ 
 Imperial College,
 South Kensington Campus,
  London SW7 2AZ,
  UK}
  \email{f.bianchi$@$imperial.ac.uk}
\end{author}

\begin{author}[Y.~Okuyama]{Y\^usuke Okuyama}
\address{Division of Mathematics, Kyoto Institute of Technology, Sakyo-ku, Kyoto 606-8585 JAPAN}
\email{okuyama$@$kit.ac.jp}
\end{author}


\begin{abstract}
{For an algebraic family $(f_t)$ 
of regular quadratic polynomial endomorphisms of $\bC^2$
parametrized by $\bD^*$ and degenerating to a H\'enon map at $t=0$, 
we study
the continuous {(and indeed harmonic)} extendibility
across $t=0$ of a potential of the bifurcation current on $\bD^*$ 
with the explicit 
{computation of the}
non-archimedean Lyapunov exponent
associated to $(f_t)$. The
individual Lyapunov exponents of $f_t$ are also
investigated near $t=0$.
Using $(f_t)$, 
we also see that any H\'enon map is accumulated by the bifurcation locus 
in the space of quadratic holomorphic endomorphisms of $\bP^2$. }
\end{abstract}

\maketitle

\section{Introduction}

Our aim is to study an algebraic family
of regular quadratic polynomial endomorphisms 
of $\bC^2$ parametrized by a punctured open disk
and degenerating to a H\'enon map 
at the puncture,
paying a particular
attention to the asymptotic behaviour of 
the 
individual Lyapunov exponents and their sums.
For one dimensional meromorphic families of rational functions, such degenerations 
towards rational functions of lower (topological) degrees
have been studied in \cite{DeMarco-stable,fg_continuity,dmo_discontinuity}.
Favre recently introduced a general framework to study such
degenerations for holomorphic endomorphisms of $\P^k$ 
in \cite{favre_degeneration} and
we here provide the first concrete study
in dimension $k$ higher than 1.
We also {study} the geometry of the \emph{bifurcation locus},
in the sense of \cite{bbd}, 
in {the space of quadratic holomorphic endomorphisms of $\P^2$}.
The geometry of the bifurcation locus 
near the line
at infinity of the moduli space $\mathcal{M}_2\cong\bC^2$ 
of the quadratic rational functions on $\bP^1$ and that near
the hyperplane at infinity
of
the natural parameter space {$\cong\bC^3$}
of quadratic polynomial skew products on $\bC^2$
have been studied in \cite{bg_geometry} and \cite{ab_skew}, respectively.

Let us be more specific and precisely state our results. In the rest of this article,
we fix 
\begin{gather*}
 c\in\bC^*\quad\text{and}\quad p(w)=w^2+c_1w+c_2\in\bC[w].
\end{gather*}
For each 
\begin{gather}\label{eq_gh}
 (g,h)\in\bC[z,w]\times\bC[z,w]\quad\text{such that } 
\deg g=2,\ g_{zz}\in\bC^*, \text{ and }\deg h\le 2,
\end{gather}
we focus on the  {algebraic} family 
\begin{equation}
 f_t(z,w)=f_t(z,w;g,h)=\begin{pmatrix}
	   w\\
	   cz+p(w)
	  \end{pmatrix}+t\begin{pmatrix}
			  g(z,w)\\
			  h(z,w)
			 \end{pmatrix},\quad t\in\bD,\label{eq_ft}
\end{equation}
of 
quadratic polynomial endomorphisms of $\C^2$ parametrized by $\bD$;
we also set the constants ${g_{zz}}/2!=: G_g = G\in\bC^*$ 
and ${h_{zz}}/2!=: H_h = H\in\bC$, respectively, and set
\begin{gather*}
 \tilde{g}(z,w):=g(z,w)-Gz^2. 
\end{gather*}

For $t=0$, the map $f_0(z,w)=(w,cz+p(w))$ is a H\'enon map 
 (see e.g.\ \cite{FM89}) and is independent of $(g,h)$. 
For every $0<|t|\ll 1$,
$f_t$ satisfies the condition 
\begin{gather}
 \liminf_{\|(z,w)\|\to\infty}\frac{\|f_t(z,w)\|}{\|(z,w)\|^2}>0\label{eq:regular} 
\end{gather}
(here $\|\cdot\|$ is the Euclidean norm on $\bC^2$),
or equivalently, $f_t$ is a {\itshape regular} quadratic polynomial 
endomorphism of $\bC^2$, that is, it extends to a holomorphic endomorphism of $\P^2$
{(see e.g.\ \cite{BedfordJonsson00})}.
In particular, if $0<|t|\ll 1$,
 then
$f_t$ admits the {(second)} Julia set $J_{f_t}$, 
which is contained in $\bC^2$ and
coincides with the support of the unique maximal entropy measure 
$\mu_{f_t}$ of (the holomorphic extension to $\bP^2$ of) $f_t$,
and
the sum $L(f_t)$ of the two {\itshape individual} Lyapunov exponents 
{$\chi_1(f_t)\ge\chi_2(f_t)$(indeed $\ge\log\sqrt{2}$ \cite{briendduval})} 
of 
$f_t$ with respect to $\mu_{f_t}$ is given by 
\begin{gather}
L(f_t)= \int_{\bC^2}\log|\det Df_t|\mu_{f_t}\in\bR.
\label{eq_lyapunov}
\end{gather}
We also call $L(f_t)$ {\itshape the Lyapunov exponent} of $f_t$ with respect to $\mu_{f_t}$.
Here and below, we fix the trivialization of
the tangent bundle $T\bC^2$ of $\bC^2$
induced by the orthonormal frame $(\partial_z,\partial_w)$
of $T\bC^2$, and identify the derivative $df_t$ of $f_t$ 
with the $M(2,\bC)$-valued function $(z,w)\mapsto\det(Df_t)_{(z,w)}$, by convention.

We regard the set of all $(g,h)$ as in \eqref{eq_gh} 
as $(\bC^*\times\bC^2)\times\bC^3$, parametrizing it 
by the coefficients of $g,h$
to mention the local uniformity of the estimates. 

\subsection{Degeneration of the Lyapunov exponent}
Our first interest is in
the asymptotic behaviour of $L(f_t)$ as $t\to{0}$,
where $f_t=f_t(z,w;g,h)$, for each $(g,h)$ as in \eqref{eq_gh}.
Such a behaviour has been studied for meromorphic families of rational functions on $\bP^1$ by DeMarco \cite{DeMarco-stable}.
In our situation, it follows from Favre's generalization
\cite{favre_degeneration} of DeMarco's estimate
that there 
is a {non-negative} constant $\alpha$ such that
\[
L(f_t) = \alpha \log |t|^{-1} + o\bigl(\log|t|^{-1}\bigr)\quad\text{as }t\to 0 
\]
and that the constant $\alpha$ is characterized 
as the \emph{non-archimedean} Lyapunov exponent 
associated to the meromorphic family $(f_t)_{t\in\bD^*}$,
regarded as a single rational function 
defined over a field of formal Laurent series at $t=0$. 
{The 
function $t\mapsto L(f_t)-\alpha\log|t|^{-1}$ is
continuous and subharmonic on $0<|t|\ll 1$ 
(see e.g.\ \cite{ds_cime} for more details), and
is a potential of the {\itshape bifurcation current} {(indeed {\itshape measure})}
on $0<|t|\ll 1$
associated to the family $(f_t)$ in the sense of \cite{bbd}.}

In the following, it is convenient to say that
the pair $(g,h)$ 
{(or the associated $(G,H)=(G_g,H_h)$)}
is {\itshape non-exceptional} if
\begin{gather}
 \biggl|\frac{H}{G}\biggr|\neq|c|.
\end{gather}
Our first principal result
answers affirmatively
Favre's general question \cite[Problem 1]{favre_degeneration}
in our context
by establishing the continuous
(and indeed harmonic)
 extendibility of
the potential $t\mapsto L(f_t)-\alpha\log|t|^{-1}$ across $t=0$,
with the concrete value $\alpha=1/2$, for any non-exceptional $(g,h)$.

\begin{mainth}\label{teo_1}
$($i$)$ Pick $(g_0,h_0)$ as in \eqref{eq_gh}.
Then for every $\beta>0$ small enough, we have
\begin{gather}
\log\frac{(1-\beta)\cdot 4\bigl||H/G|-|c|\bigr|^{1/2}}{|G|^{1/2}}
\le L(f_t)-\frac{1}{2}\log|t|^{-1}\le
\log\frac{(1+\beta)\cdot 4\bigl(|H/G|+|c|\bigr)^{1/2}}{|G|^{1/2}}\label{eq_prelim}
\end{gather}
for every $t\in\bD^*$ close enough to $0$ and
every 
$(g,h)$ close enough to $(g_0,h_0)$, recalling
that 
$f_t=f_t(z,w;g,h)$ and $(G,H)=(G_g,H_g)$.
In particular, for every $(g,h)$,
the non-archimedean Lyapunov exponent $\alpha$ 
associated to the meromorphic family $(f_t)_{t\in\bD^*}$ equals $1/2$.

$($ii$)$ {Pick a non-exceptional $(g,h)$. Then
for the algebraic family $(f_t)_{t\in\bD^*}$
in \eqref{eq_ft} associated to this $(g,h)$}, 
the continuous and subharmonic 
function $t\mapsto L(f_t)-(1/2)\log|t|^{-1}$ 
is harmonic on $0<|t|\ll 1$ 
and extends harmonically
across $t=0$, satisfying
\begin{gather}
\lim_{t\to 0}\Bigl(L(f_t)-\frac{1}{2}\log|t|^{-1}\Bigr)
=\log\frac{4\max\bigl\{|c|,|H/G|\bigr\}^{1/2}}{|G|^{1/2}}.\label{eq:lim}
\end{gather}
\end{mainth}

Notice that a similar continuous extendability result
has been obtained
for meromorphic families of polynomials in one variable
by Favre--Gauthier \cite{fg_continuity} and that
examples of discontinuity at $t=0$ have been obtained
in \cite{dmo_discontinuity}
for meromorphic families of rational functions on $\bP^1$.

\subsection{Accumulation of the bifurcation locus 
to the H\'enon locus}
For a holomorphic family of rational functions on $\bP^1$,
the theory of $J$-stability/bifurcation, originating from the seminal papers by 
Ma\~{n}\'e-Sad-Sullivan \cite{mss}, Lyubich \cite{lyubich}, 
DeMarco \cite{demarco1,demarco2}, is now classical. 
A generalization of this theory to holomorphic families of
holomorphic endomorphisms of $\P^k$ 
parametrized by a complex manifold $M$ was recently
developed in
\cite{bbd,b_misiurewicz}. 
There the \emph{bifurcation locus} in the parameter space $M$
is {\itshape defined} 
as
the support of 
the \emph{bifurcation current} on $M$; the bifurcation current on $M$ is
the $dd^c$ of the Lyapunov exponent function
$\lambda\mapsto L(f_\lambda):=\int_{\bP^k}\log|\det Df_{\lambda}|\mu_{f_{\lambda}}$ on $M$,
where $\mu_{f_\lambda}$
is the unique maximal entropy measure of $f_\lambda$ on $\bP^k$. 
In dimension $k=2$, if in addition $M$ is simply connected, 
then 
the bifurcation locus coincides with, e.g., the complement in $M$
of the locus where
all
the repelling cycles of $f_\lambda$
in the 
second
Julia sets $J_{f_\lambda}$ of $f_\lambda$ move holomorphically.
We refer to \cite{bbd} for more details, and to \cite{b_misiurewicz} for
an analogous (slightly weaker)
characterization
valid in any dimension 
$k$.

Let us now focus on the family
\begin{gather}
f_{t,G,H} (z,w):=
\begin{pmatrix}
	   w\\
	   cz+p(w)
	  \end{pmatrix}+t\begin{pmatrix}
			  Gz^2\\
			  H z^2
			 \end{pmatrix},
\quad (t,G,H)\in\bD^*\times\bC^*\times\bC,
\label{eq_ftGH}
\end{gather}
of quadratic polynomial endomorphisms of $\bC^2$,
which are regular for every $t\in\bD^*$
since the leading homogeneous term
$(Gz^2,w^2+Hz^2)$ of $f_{t,G,H}$ maps only $(0,0)$ to $(0,0)$.

The following is our second principal result. 

\begin{mainth}\label{teo_2}
The bifurcation locus in the parameter space {$\bD^*\times\bC^*\times\bC$}
of the family $(f_{t,G,H})$ 
accumulates to $\{t=0\}$ in $\bD\times\bC^*\times\bC$
tangentially to the locus $|H/G| = |c|$ in $\bD^*\times\bC^*\times\bC$.
\end{mainth}

The set $\operatorname{Hol}_2(\P^2)$ of all quadratic holomorphic endomorphisms of $\bP^2$
is a Zariski open subset in $\bP^{N_2}$,
where $N_2=3\cdot 4!/(2!2!)-1=17$, in the coefficients parametrization,
and in turn is regarded as {the parameter space of}
the holomorphic family of (all) quadratic holomorphic endomorphisms
of $\bP^2$. We also note that 
all H\'enon maps live
in $\bP^{N_2}\setminus{\operatorname{Hol}_2(\P^2)}$ in the coefficients parametrization.

The following immediate consequence of Theorem \ref{teo_2} is also
one of our principal results,
and answers affirmatively a question by Johan Taflin. 

\begin{maincoro}\label{cor_bif_henon}
The H\'enon locus in $\bP^{N_2}\setminus{\operatorname{Hol}_2(\P^2)}$
is accumulated by the bifurcation locus of {$\operatorname{Hol}_2(\P^2)$}.
\end{maincoro}

The proof of Theorem \ref{teo_2}
(so that of Corollary) is based on Theorem \ref{teo_1}
and is
purely analytical. In former studies,
the presence of bifurcations 
 has
been established 
by means of more geometric arguments; 
see e.g.\ \cite{bt_desboves, dujardin_nondensity, bb_hausdorff, taflin_blender, ab_skew,biebler_lattes}. 

\subsection{Individual Lyapunov exponents}
{It follows from a result by Pham \cite{pham} that, 
for any holomorphic family $(f_t)_{t\in M}$ parametrized
by a complex manifold $M$ 
of holomorphic endomorphisms of $\P^k$ of a given degree $d>1$,
if we denote by $\chi_1(f_t) \geq \dots \geq \chi_k(f_t)$
all the individual Lyapunov exponents of $f_t$ 
for each $t\in M$, then for every $j\in\{1,\ldots,k\}$, the function
$t\mapsto \sum_{\ell=1}^j\chi_\ell(f_t)$ on $M$ is plurisubharmonic.}

Let us focus on our family $(f_t)$ as in \eqref{eq_ft}. 
{Recall that the function 
$t\mapsto L(f_t)\equiv\chi_1(f_t)+\chi_2(f_t)$
is continuous and subharmonic on $0<|t|\ll 1$, so
by the above result by Pham, the function $t\mapsto\chi_1(f_t)$ is 
subharmonic and the function $t\mapsto\chi_2(f_t)$ is lower semicontinuous,
in general.}
We conclude this introduction with
the following precision of
Theorem \ref{teo_1}(ii). 

\begin{mainth}\label{th:individual}
Pick a non-exceptional $(g,h)$.
Then
for the algebraic family $(f_t)_{t\in\bD^*}$
{in \eqref{eq_ft} associated to this $(g,h)$}, 
the functions $t\mapsto\chi_1(f_t)- \frac{1}{2}\log|t|^{-1}$ and 
$t\mapsto\chi_2(f_t)$ are {harmonic} on $0<|t|\ll 1$
and extend
harmonically
across $t=0$, satisfying
\begin{gather*}
\lim_{t\to 0}\Bigl(\chi_1(f_t)-\frac{1}{2}\log|t|^{-1}\Bigr)=\log\frac{2\max\bigl\{|c|,|H/G|\bigr\}^{1/2}}{|G|^{1/2}
}
\quad\text{and}\quad\lim_{t\to 0}\chi_2(f_t)=\log 2. 
\end{gather*}
\end{mainth}

The proof of the harmonicity of $\chi_1(f_t),\chi_2(f_t)$ 
is based on {the full strength of}
Berteloot--Dupont--Molino's approximations, {that is, approximations of} 
{not only $\chi_1(f_t)+\chi_2(f_t)(=L(f_t))$ but also $\chi_1(f_t)$} 
\cite[Theorem 1.5]{BDM08}.

\subsection{Organization of the article}
In Section \ref{sec:key}, we establish a key estimate, 
which is a development of an estimate appearing in \cite{dujardin_nondensity}.
In Section \ref{sec:proofconti}, we show Theorem \ref{teo_1}.
In Section \ref{section_bif}, we show Theorem \ref{teo_2} (and Corollary)
using Theorem \ref{teo_1}(i).
We also include a comparison of our analysis here with the
study of bifurcations of quadratic polynomial skew products in \cite{ab_skew}.
In Section \ref{sec:individual}, we show Theorem \ref{th:individual},
which is a precision of Theorem \ref{teo_1}(ii). To make such a precision,
we also recall some standard facts from ergodic theory 
as well as  Berteloot--Dupont--Molino's approximations.

\section{A key lemma}
\label{sec:key}

In this section we show {a lemma} concerning the position
of the (second) Julia set $J_{{f_t}}$ of $f_t$, 
{which is} 
needed in the sequel and 
inspired by \cite[Lemma 5.2]{dujardin_nondensity}. From now on,
set $B(r):=\{z\in\bC:|z|<r\}$ for every $r>0$ and
$A(r,s):=\{z\in\bC:r<|z|<s\}$ for any $r,s\in\bR$ satisfying $0<r<s$;
as convention, we also set $A(0,s):=B(s)$ for every $s>0$.

\begin{lemma}\label{lemma_stime}
Pick $(g,h)=(g_0,h_0)$ {as in \eqref{eq_gh}}. Then
there is $\beta\in(0,{1/2})$ so small that
for every $(g,h)$ close enough to $(g_0,h_0)$ 
and every $t\in\bD^*$ close enough to $0$, 
recalling {that}
$f_t=f_t(z,w;g,h)$ and $(G,H)=(G_g,H_g)$
and
setting 
\begin{gather*}
\begin{aligned}
U_t=U_t(\beta;g,h)& :=A\biggl(\frac{1-\beta}{|Gt|},\frac{1+\beta}{|Gt|}\biggr)\times 
B\biggl(\frac{\bigl(|H/G|+|c|\bigr)^{1/2}(1+2\beta)}{|Gt|^{1/2}}\biggr)\quad\text{and}\\
V_t = V_t (\beta;g,h) &:=
A\biggl(\frac{1-\beta}{|Gt|},\frac{1+\beta}{|Gt|}\biggr)\times
A\biggl(
\frac{\bigl||H/G|-|c|\bigr|^{1/2}(1-2\beta)}{|Gt|^{1/2}},
\frac{\bigl(|H/G|+|c|\bigr)^{1/2}(1+2\beta)}{|Gt|^{1/2}}\biggr),
\end{aligned}
\end{gather*}
we have $f_t^{-1}(U_t)\Subset V_t$, and in particular, $J_{f_t}\subset V_t$. 
\end{lemma}

\begin{proof}
{Pick $(g_0,h_0)$}. 
Let us see the former assertion. Suppose to the contrary that there exist
\begin{itemize}
 \item a sequence $(\beta_n)$ in $(0,1)$ tending to $0$ as $n\to\infty$, 
 \item a sequence $((g_n,h_n))$ tending to $(g_0,h_0)$ as $n\to\infty$,
 \item a sequence $(t_n)$ in $\bD^*$ tending to $0$ as $n\to\infty$, and 
 \item a sequence $((z_n,w_n))$ in $\bC^2$ 
\end{itemize}
such that for every $n\in\bN$, 
{we have $(z_n,w_n)\in\bC^2\setminus V_{t_n}(\beta_n;g_n,h_n)$ 
 and}
\begin{gather*}
 (u_n,v_n):=f_{t_n}(z_n,w_n;g_n,h_n)\in\overline{U_{t_n}(\beta_n;g_n,h_n)}.
\end{gather*}

\begin{claim}
As $n\to\infty$, $w_n+t_n\tilde{g}_n(z_n,w_n)=o\bigl(t_nz_n^2\bigr)$. 
\end{claim} 

\begin{proof}
Otherwise,
taking a subsequence if necessary, 
there exists $C>0$ such that for every $n\in\bN$,
\begin{gather}
 |w_n+t_n\tilde{g}_n(z_n,w_n)|\ge C|t_nz_n^2|.\label{eq:contradiction} 
\end{gather}
Let us first see that, taking a further subsequence if necessary,
there exists $C'>0$ such that for every $n\in\bN$,
\begin{gather}
\max\{|w_n|,|t_n z_n w_n |, |t_n w_n^2 |\}
\ge C' \max \{|t_n z_n^2|,|t_n|^{-1}\};\tag{\ref{eq:contradiction}$'$}\label{eq_max_max}
\end{gather}
indeed, {taking a subsequence if necessary, there are exactly two possibilities;}

{\bfseries (a)} if $|t_n z_n^2|\ge |t_n|^{-1}$ for every $n\in\bN$,
then also by \eqref{eq:contradiction},
we have $|w_n+t_n\tilde{g}_n(z_n,w_n)|\ge C\max\{|t_nz_n^2|,|t_n|^{-1}\}$,
so if $n\gg 1$, then 
$\max\{1,|g_{zw}|,|g_{ww}|/2\}\max\{|w_n|,|t_nz_nw_n|,|t_nw_n^2|\}\ge C\max\{|t_nz_n^2|,|t_n|^{-1}\}$ (recall $\tilde{g}_n(z,w):=g_n(z,w)-G_nz^2$ and $g_{zw},g_{ww}\in\bC$), which yields \eqref{eq_max_max}
in this case.

{\bfseries (b)} If {$|t_n z_n^2|\le|t_n|^{-1}$ for every $n\in\bN$} but,
to the contrary, 
$\max\{|w_n|,|t_n z_n w_n |, |t_n w_n^2 |\}=o(t_n^{-1})$ as $n\to\infty$,
then {$|z_n|\le|t_n|^{-1}$ for every $n\in\bN$, and then 
under the assumption \eqref{eq:contradiction},}
recalling 
$\tilde{g}_n(z,w):=g_n(z,w)-G_nz^2$,
{we also have $t_nz_n^2=o(t_n^{-1})$ as $n\to\infty$,
that is, $z_n=o(t_n^{-1})$ as $n\to\infty$.
Then we must have
$u_n=w_n+t_ng(z_n,w_n)=o(t_n^{-1})$ as $n\to\infty$,}
which contradicts $(u_n,v_n)\in\overline{U_{t_n}(\beta_n;g_n,h_n)}$
for every $n\in\bN$. Hence \eqref{eq_max_max} also holds in this case.
 
Once \eqref{eq_max_max} is at our disposal, we can
{deduce}
a contradiction
as follows. Taking a subsequence if necessary, there are exactly three possibilities.

{\bfseries (1)}
If $|w_n|\ge\max\{|t_nz_nw_n|,|t_nw_n^2|\}$ 
for every $n\in\bN$, then by \eqref{eq_max_max}, we have 
$|w_n|\ge C'|t_n|^{-1}$ and, moreover,
$|z_n|^2\le |t_n^{-1}w_n|/C'\le(|w_n|/C')^2$, 
that is, $|z_n|\le|w_n|/C'$ for every $n\in\bN$. 
Then $v_n:=cz_n+p(w_n)+t_{{n}}h(z_n,w_n)
= (1+o(1))w_n^2$ as $n\to\infty$, so 
by $(u_n,v_n)\in\overline{U_{t_n}(\beta_n;g_n,h_n)}$,
we have $w_n=O(|t_n|^{-1/4})$ as $n\to\infty$. 
{Then we must have $0<C'\le|t_n^{-1}w_n|=O(|t_n|^{3/4})\to 0$ as $n\to\infty$,
which is impossible.}

{\bfseries (2)} 
If $|t_nz_nw_n|\ge\max\{|w_n|,|t_nw_n^2|\}$ for every $n\in\bN$, then 
for every $n\in\bN$,
we have {$|z_n|\ge|t_n|^{-1}$ 
and,}
by \eqref{eq_max_max}, also have $|z_n|\le|w_n|/C'$, 
so that $|t_nz_n^2|\le|t_nw_n^2|/(C')^2$ and $|t_nz_nw_n|\le|t_n w_n^2|/C'$.
Then $v_n:=cz_n+p(w_n)+t_{{n}}h(z_n,w_n)=(1+o(1))w_n^2$ as $n\to\infty$, so 
by $(u_n,v_n)\in\overline{U_{t_n}(\beta_n;g_n,h_n)}$,
we have $w_n = O(|t_n|^{-1/4})$ as $n\to\infty$. 
{Then we must have 
$1\le|t_nz_n|\le|t_nw_n|/C'=O(|t_n|^{3/4})\to 0$ as $n\to\infty$, which is impossible.}
 
{\bfseries (3)} 
If $|t_nw_n^2|\ge\max\{|w_n|,|t_nz_nw_n|\}$ for every $n\in\bN$, then 
for every $n\in\bN$, we have {$|w_n|\ge|t_n|^{-1}$, and}
by \eqref{eq_max_max}, also have $|w_n|\ge\sqrt{C'}|z_n|$, 
so that $|t_nz_n^2|\le|t_nw_n^2|/C'$ and $|t_nz_nw_n|\le|t_n w_n^2|/\sqrt{C'}$. 
Then $v_n:=cz_n+p(w_n)+t_{{n}}h(z_n,w_n)=(1+o(1))w_n^2$ as $n\to\infty$, so 
by $(u_n,v_n)\in\overline{U_{t_n}(\beta_n;g_n,h_n)}$,
we have $w_n = O(|t_n|^{-1/4})$ as $n\to\infty$. 
{Then we must have $1\le|t_nw_n|=O(|t_n|^{3/4})\to 0$ as $n\to\infty$,
which is impossible.}
Hence the claim holds.
\end{proof}

For every $n\in\bN$ large enough,
 by the equality 
$u_n=w_n+G_nt_nz_n^2+t_n\tilde{g}_n(z_n,w_n)$ and the above Claim, we have
$z_n=(1+o(1))(u_n/(G_n t_n))^{1/2}$ as $n\to\infty$, so that
{recalling that} $(u_n,v_n)\in\overline{U_{t_n}(\beta_n;g_n,h_n)}$ 
and $\beta_n{\searrow 0}$, we have
\begin{gather*}
z_n\in A\biggl(\frac{(1+o(1))(1-{\beta_n})^{1/2}}{|G_nt_n|},\frac{(1+o(1))(1+{\beta_n})^{1/2}}{|G_nt_n|}\biggr)\subset
A\biggl(\frac{1-{\beta_n}}{|G_nt_n|},\frac{1+{\beta_n}}{|G_nt_n|}\biggr).
\end{gather*}
Moreover,
{since}
$v_n=cz_n+p(w_n)+t_{n} h(z_n,w_n)$ and
$(u_n,v_n)\in\overline{U_{t_n}(\beta_n;g_n,h_n)}$ for every $n\in\bN$, we have
\begin{gather*}
 w_n^2+c_1w_n=v_n-t_nh(z_n,w_n)-cz_n-c_2
=-t_n H_n z_n^2-cz_n+o(t_n^{-1})\quad\text{as }n\to\infty.
\end{gather*}
{Since}
$(z_n,w_n)\not\in V_{t_n}{(\beta_n;g_n,h_n)}$ for every $n\in\bN$,
taking a subsequence if necessary, {there are two possibilities};

{\bfseries (i)}
if $|w_n|\le\bigl||H_n/G_n|-|c|\bigr|^{1/2}(1-2\beta_n)/|G_nt_n|^{1/2}$ 
for every $n\in\bN$, then $w_n=O(t_n^{-1/2})$ as $n\to\infty$, and then
$w_n^2=-t_n H_n z_n^2-cz_n+o(t_n^{-1})$ as $n\to\infty$.

{\bfseries (ii)}
Alternatively, if $|w_n|\ge(|H_n/G_n|+|c|)^{1/2}(1-2\beta_n)/|G_nt_n|^{1/2}$
for every $n\in\bN$, then $w_n=o(w_n^2)$ as $n\to\infty$, 
and then $w_n^2=(-t_n H_n z_n^2-cz_n+o(t_n^{-1}))/(1+o(1))$ as $n\to\infty$.

So in any case, 
\begin{gather*}
 w_n^2=(1+o(1))(-t_n H_n z_n^2-cz_n+o(t_n^{-1}))\quad\text{as }n\to\infty.
\end{gather*}
Hence
{recalling that}
$\beta_n{\searrow 0}$, we also have 
\begin{gather*}
\frac{\bigl||H_n/G_n|-|c|\bigr|(1-2{\beta_n})^2}{|G_nt_n|}
{<} |w_n|^2
{<} \frac{\bigl(|H_n /G_n|+|c|\bigr)(1+2{\beta_n})^2}{|G_nt_n|}
\end{gather*}
for every $n\in\bN$ large enough. Hence we must have
$(z_n,w_n)\in V_{t_n}(\beta_n;g_n,h_n)$. This {gives the desired} contradiction.

Once the former assertion is at our disposal, the latter assertion 
follows from, e.g.,
the fact that $\mu_{f_t} = \lim_{n\to \infty}2^{-2n}(f_t^n)^*\Omega$
weakly on $\bC^2$ 
for any smooth probability measure $\Omega$ compactly supported in $\C^2$
(see for instance \cite{ds_cime}).
\end{proof}

\section{Proof of Theorem \ref{teo_1}}\label{sec:proofconti}

{Let us see the former assertion (i).}
Pick $(g,h)=(g_0,h_0)$ {as in \eqref{eq_gh}}.
By Lemma \ref{lemma_stime}, 
we can fix $\beta\in(0,1/2)$ so small
that for every 
$(g,h)$ close enough to $(g_0,h_0)$
and every $t\in\bD^*$ close enough to $0$,
recalling that
$f_t=f_t(z,w;g,h)$,
we have $J_{f_t}\subset V_t {(=V_t(\beta;g,h))}$, so that
\begin{gather*}
  L(f_t) = \int_{V_t}\log|\det Df_t|\mu_{f_t}(z,w).
\end{gather*}
For every $(z,w)\in\bC^2$, 
we compute as
\begin{align*}
\det (Df_t)_{(z,w)}
 & =  {\det\begin{pmatrix}
      tg_z & 1+tg_w\\
      c+th_z & p'(w)+th_w
     \end{pmatrix}}
 =  {\det\begin{pmatrix}
		2tGz+t\tilde{g}_z & 1+tg_w\\
      c+th_z & (2w+c_1)+th_w
	       \end{pmatrix}}\\
 & = 4tGzw+2tGzc_1+{t\tilde{g}_z(p'(w)+th_w)}-(1+tg_w)(c+th_z)\\
 & = 4tGzw\cdot\biggl(1+\frac{2tGzc_1+{t\tilde{g}_z(p'(w)+th_w)}-(1+tg_w)(c+th_z)}{4tGzw}\biggr).
\end{align*}
From the definition of $V_t$, we have 
\begin{gather*}
\frac{4(1-\beta)\bigl||H/G|-|c|\bigr|^{1/2}(1-2\beta)}{|G|^{1/2}|t|^{1/2}}
\le|4tGzw|\le
\frac{4(1+\beta)\bigl(|H/G|+|c|\bigr)^{1/2}(1+2\beta)}{|G|^{1/2}|t|^{1/2}}\quad\text{on }V_t,
\end{gather*}
and moreover, $\sup_{V_t}
|\det (Df_t)_{(z,w)}-4tGzw|=O(1)$ as $t\to 0$,
uniformly on $(g,h)$
near $(g_0,h_0)$.
Hence increasing $\beta$ slightly if necessary, we have
\begin{gather}
\frac{4(1-\beta)\bigl||H/G|-|c|\bigr|^{1/2}(1-2\beta)}{|G|^{1/2}|t|^{1/2}}
\le|\det(Df_t)_{(z,w)}|
\le\frac{4(1+\beta)\bigl(|H/G|+|c|\bigr)^{1/2}(1+2\beta)}{|G|^{1/2}|t|^{1/2}}
\quad\text{on }V_t
\label{eq:hyperbolic}
\end{gather}
for every $t\in\bD^*$ close enough to $0$ and
every $(g,h)$ close enough to $(g_0,h_0)$. In particular,
\eqref{eq_prelim} holds, increasing $\beta$ slightly if necessary.

Now let us see the convergence assertion in the
latter assertion (ii).
Pick a {non-exceptional} $(g,h)$.
From the above computation of $\det(Df_t)_{(z,w)}$, we also have
\begin{gather*}
L(f_t)-\frac{1}{2}\log|t|^{-1}=\int_{V_t}\log|4tGzw|\mu_{f_t}(z,w)-\frac{1}{2}\log|t|^{-1}+o(1)\quad\text{as }t\to 0,
\end{gather*}
and we
 compute as 
\begin{align*}
 & \int_{V_t}\log|tzw|\mu_{f_t}(z,w)-\frac{1}{2}\log|t|^{-1}
=-\frac{3}{2}\log|t|^{-1}
+\int_{V_t}\log|z|\mu_{f_t}(z,w)+\int_{V_t}\log|w|\mu_{f_t}(z,w)\\
=&-\frac{3}{2}\log|t|^{-1}
+\int_{V_t}\log|z|\mu_{f_t}(z,w)+\frac{1}{2}\int_{V_t}\log|w|^2\mu_{f_t}(z,w)\\
= &- \frac{3}{2}\log|t|^{-1}+
\int_{V_t}\log|z|\mu_{f_t}(z,w)
+\frac{1}{2}\int_{V_t}\biggl(\log|cz+tHz^2|+\log\left|1+\frac{o(t^{-1})}{cz+tHz^2}\right|\biggr)\mu_{f_t}(z,w)\\
= &  \frac{3}{2}\Bigl(\int_{V_t}\log|z|\mu_{f_t}(z,w)-\log|t|^{-1}\Bigr)\\
&+\frac{1}{2}\int_{V_t}\log|c+tHz|\mu_{f_t}(z,w)
+\frac{1}{2}\int_{V_t}\log\left|1+\frac{o(t^{-1})}{cz+tHz^2}\right|\mu_{f_t}(z,w),
\end{align*}
where setting $(u,v):= f_t (z,w)\in J_{f_t}\subset V_t$, we also used the estimate
\begin{gather*}
 w^2=-th(z,w)-cz+v-c_1w-c_2=-tHz^2-cz {+o(t^{-1})\quad\text{as }t\to 0}.
\end{gather*}
Since $(1-\beta)\le|tGz|\le (1+\beta)$ on $V_t$, we have
\begin{gather*}
\int_{V_t}\log|z|\mu_{f_t}(z,w)-\log|t|^{-1}=-\log|G| {+o(1)\quad\text{as }t\to 0}
\end{gather*}
(as $\beta\to 0$). Similarly, 
{under the assumption that $(g,h)$ is non-exceptional,}
since
\begin{gather*}
0<\frac{\bigl||c|-|H/G|\bigr|}{|G|}
\le\liminf_{t\to 0}\biggl|\frac{cz+tHz^2}{t^{-1}}\biggr|
{\le\limsup_{t\to 0}\biggl|\frac{cz+tHz^2}{t^{-1}}\biggr|\le\frac{|c|+|H/G|}{|G|}}
\end{gather*}
(as $\beta\to 0$),
we have
\begin{gather*}
\int_{V_t}\log\left|1+\frac{o(t^{-1})}{cz+tHz^2}\right|\mu_{f_t}(z,w) {= o(1)\quad\text{as }t\to 0}.
\end{gather*}
It remains to show that
$\int_{V_t}\log|c+tHz|\mu_{f_t}(z,w) =\log\max\{|c|,|H/G|\} { +o(1)}$ as $t\to 0$.
Under the (linear) coordinates system 
change $(z,w)\mapsto(Z,W)=(Gtz,tw)$ on $\bC^2$, setting
\begin{align*}
\tilde{f}_t(Z,W)  :=&\begin{pmatrix}
	    Z\\
	    W
\end{pmatrix}\circ f_t\circ \begin{pmatrix}
				   Z\\
				   W
				  \end{pmatrix}^{-1}
=\begin{pmatrix}
 Z^2+GW+Gt^2\tilde{g}\Bigl(\frac{Z}{Gt},\frac{W}{t}\Bigr)\\
 {t}p\Bigl(\frac{W}{t}\Bigr)+c\frac{Z}{Gt}+th\Bigl(\frac{Z}{Gt},\frac{W}{t}\Bigr)
\end{pmatrix},\\
\tilde{\mu}_t  :=&\mu_{\tilde{f}_t}=\begin{pmatrix}
	    Z\\
	    W
	   \end{pmatrix}_*\mu_{f_t},\\
\tilde{U}_t  :=&\begin{pmatrix}
	    Z\\
	    W
	   \end{pmatrix}(U_t)
 =A(1-\beta,1+\beta)\times B\biggl(\frac{\bigl(|H/G|+|c|\bigr)^{1/2}(1+2\beta)}{|G|^{1/2}}|t|^{1/2}\biggr),\quad\text{and}\\
 \tilde{V}_t  :=&\begin{pmatrix}
	    Z\\
	    W
	   \end{pmatrix}(V_t)\\
  =&A(1-\beta,1+\beta)\times A\biggl(
 \frac{\bigl||H/G|-|c|\bigr|^{1/2}(1-2\beta)}{|G|^{1/2}}|t|^{1/2},
 \frac{\bigl(|H/G|+|c|\bigr)^{1/2}(1+2\beta)}{|G|^{1/2}}|t|^{1/2}\biggr)
\end{align*}
and letting $p_1:(Z,W)\mapsto Z$ be the projection to the first coordinate, we compute as
\begin{multline*}
\int_{V_t}\log|c+tHz|\mu_{f_t}(z,w)
=\int_{\tilde{V}_t}\begin{pmatrix}
	    Z\\
	    W
	   \end{pmatrix}_*(\log|c+tHz|\mu_{f_t}(z,w))\\
=\int_{\tilde{V}_t}\log\biggl|c+\frac{H}{G}Z\biggr|
\left(\begin{pmatrix}
	    Z\\
	    W
	   \end{pmatrix}_*\mu_{f_t}\right)(Z,W)
=\int_{\tilde{V}_t}\log\biggl|c+\frac{H}{G}Z\biggr|\tilde{\mu}_t(Z,W)\\
=\int_{A(1-\beta,1+\beta)}\log\biggl|c+\frac{H}{G}Z\biggr|((p_1)_*\tilde{\mu}_t)(Z).
\end{multline*}
Set $S^1_Z:=\{Z\in\bC_Z:|Z|=1\}$. We claim that 
$\lim_{t\to 0}(p_1)_*\tilde{\mu}_t=m_{S^1_Z}$ weakly on $\bC_Z$;
let $\nu$ be any weak limit point
of $\tilde{\mu}_t$ 
on $\bC^2$ as $t\to 0$, 
which is supported by $S^1_Z$ 
(as $\beta\to 0$).
Set $\tilde{D}_t:=\tilde{U}_t\cap(\bR_{>0}\times\bC)$.
For every $n\in\bN$, if $0<|t|\ll 1$, then
by \eqref{eq:hyperbolic}, 
we have $\det Df_t\neq 0$ on 
$V_t$ 
{(under the assumption that
$(g,h)$ is non-exceptional)}, so that 
$\tilde{f}_t^n:\tilde{f}_t^{-n}(\tilde{U}_t)\to\tilde{U}_t$ is an 
unbranched covering of degree $2^{2n}$,
$\tilde{f}_t^{-n}(\tilde{D}_t)$ consists of $2^{2n}$ {analytic} disks,
$\tilde{f}_t^{-n}(\tilde{U}_t\setminus\tilde{D}_t)$ consists of
$2^{2n}$ components, and for each component $U$ of 
$\tilde{f}_t^{-n}(\tilde{U}_t\setminus\tilde{D}_t)$, 
$\tilde{\mu}_t(U)=1/2^{2n}$ since $\tilde{\mu}_t(\tilde{U}_t\setminus\tilde{D}_t)=1$ and $\tilde{f}_t^*\tilde{\mu}_t=2^2\tilde{\mu}_t$ on $\bC^2$;
moreover, as $t\to 0$, $\tilde{f}_t^{-n}(\tilde{D}_t)$ tends to 
the set of all $2^n$-th roots of unity in $\bC_Z$ and
for each component $V$ of 
$S^1_Z\setminus\{2^n\text{-th roots of unity}\}$,
exactly $2^{2n}/2^n$
components
of $\tilde{f}_t^{-n}(\tilde{U}_t\setminus\tilde{D}_t)$ tend
to 
$V$
(as $\beta\to 0$). Hence for every $n\in\bN$ and
every component $V$ of $S^1_Z\setminus\{2^n\text{-th roots of unity}\}$,
we have $\nu(V)={(2^{2n}/2^n)}\cdot 1/2^{2n}=1/2^n$,
which implies that $(p_1)_*\nu=m_{S^1_Z}$ on $\bC_Z$
(by Caratheodory's 
theorem). Hence the claim holds.

Once this claim is at our disposal, recalling also $J_{f_t}\subset V_t$, we have
the desired convergence
\begin{gather*}
\lim_{t\to 0}\int_{A(1-\beta,1+\beta)}\log\biggl|c
+\frac{H}{G}Z\biggr|((p_1)_*\tilde{\mu}_t)(Z)
=\int_{\bC_Z}\log\left|c+\frac{H}{G}Z\right| m_{S^1_Z}(Z)
=\log\max\left\{|c|,\left|\frac{H}{G}\right|\right\}
\end{gather*}
since the function $Z\mapsto\log\bigl|c+(H/G)Z\bigr|$ is continuous
{near $S^1_Z$}
under the assumption $|H/G|\neq|c|$.

Finally, 
let us see the harmonicity assertion of $t\mapsto L(f_t)-(1/2)\log|t|^{-1}$ on $0<|t|\ll 1$ in the latter assertion (ii).
If $0<|t|\ll 1$, then 
by $f_t^{-1}(U_t)\Subset V_t$ and \eqref{eq:hyperbolic}
(and the assumption that $|H/G|\neq|c|$),
we have (not only $J_{f_t}\subset V_t$ but also)
$V_t\cap\bigcup_{n\in\bN\cup\{0\}}f_t^n(C_{f_t})=\emptyset$,
where $C_{f_t}:=\{p\in\bC^2:\det(Df_t)_p=0\}$ is the critical set of $f_t$.
In particular, by \cite[(F)$\Rightarrow$(B) in Theorem 1.1]{bbd},
 the function 
$t\mapsto L(f_t)$ is harmonic on $0<|t|\ll 1$. \qed

\section{Proof of Theorem \ref{teo_2}}\label{section_bif}

 Let us first note that by
 an argument similar
 to that in the final paragraph
in the proof of Theorem \ref{teo_1},
for every $(G_0,H_0)\in\bC^*\times\bC$ satisfying 
the assumption
 $|H_0/G_0|\neq|c|$, 
there exists $0<r_0\ll 1$ such that for 
every $(G,H)$ close enough to $(G_0,H_0)$ and every $t\in\bD_{r_0}^*$,
the parameter $(t,G,H)$ is not in the bifurcation locus 
in the parameter space $\bD^*\times\bC^*\times\bC$ of the family $(f_{t,G,H})$.
{In particular,
for every $G_0\in\bC^*$, every sequence $(t_n)$ in $\bD^*$
tending to $0$ as $n\to\infty$, and every {\itshape bounded}
sequence $(H_n)$ in $\bC$,
if for every $n\in\bN$, 
the function $H\mapsto L(f_{t_n,G_0,H})$ on $\bC$
is not harmonic on any open neighborhood
of $H=H_n$, then $\lim_{n\to\infty}|H_n/G_0|=|c|$.}

Let us next see that for every $G_0\in\bC^*$ {and every $R>2|cG_0|$},
if $0<|t|\ll 1$, then {the function $H\mapsto L(f_{t,G_0,H})$
is not harmonic on $\{|H|<R\}$;}
otherwise, there exist $R>2|cG_0|$ and
a sequence $(t_n)$ in $\bD^*$ 
tending to $0$ as $n\to\infty$ such that for every $n\in\bN$,
the function $H\mapsto L(f_{t_n,G_0,H})$ is harmonic on the
open disk $\{|H|<R\}$.
Then using the estimate \eqref{eq_prelim} {of $L(f_{t,g_0,0})$}
for $0<|t|\ll 1$,
the mean value theorem for the harmonic functions $H\mapsto L(f_{t_n,G_0,H})$
on $\{|H|<R\}$, and
the (lower) estimate \eqref{eq_prelim} of $L(f_{t,G_0,H})$ 
for $0<|t|\ll 1$, which holds uniformly on the circle $\{|H|=R\}$,
we must have
\begin{multline*}
\log\frac{4|c|^{1/2}}{|G_0|^{1/2}}
=\lim_{n\to\infty}\Bigl(L(f_{t_n,G_0,0})-\frac{1}{2}\log|t_n|^{-1}\Bigr)\\
=\lim_{n\to\infty}\Bigl(\int_0^{2\pi}L\Bigl(f_{t_n,G_0,Re^{i\theta}}\Bigr)\frac{d\theta}{2\pi}-\frac{1}{2}\log|t_n|^{-1}\Bigr)
\ge\log\frac{4\bigl|R/|G_0|-|c|\bigr|^{1/2}}{|G_0|^{1/2}}
>\log\frac{4|c|^{1/2}}{|G_0|^{1/2}}
\end{multline*}
(as $\beta\to 0$), which is impossible. 
Now the proof of Theorem \ref{teo_2} is complete. \qed

\medskip

We conclude this section
with a description of similarities
between Theorem \ref{teo_2} and one of the main results in \cite{ab_skew}, where
the accumulation of the bifurcation locus
to  the hyperplane at
 infinity for holomorphic families of quadratic polynomial skew products 
on $\bC^2$ has been completely described. 

\subsection*{A comparison with the bifurcation of polynomial skew products}\label{section_comparison}

Let us 
introduce in the family \eqref{eq_ftGH}
(after the coordinate change $(z,w)\mapsto(Gtz,w)$)
an extra parameter {$\eta\in\{0,1\}$} as
\begin{equation}\label{eq_comparison}
f_{t,G,H,\eta} (z,w)
:=
\begin{pmatrix}
	   z^2+\eta\cdot t G w \\
	   w^2 + \frac{H}{tG^2} z^2 + \frac{c}{tG}z
\end{pmatrix},
\end{equation}
so that
the family \eqref{eq_ftGH} (after the above coordinate change) 
corresponds
to the choice $\eta=1$ of the parameter $\eta$.
In the case $\eta=0$,
we get a family of {regular quadratic} \emph{polynomial skew products} 
on $\bC^2$ of the form studied in \cite{ab_skew}.
It is there proved by a method different from here and
based on the characterization of stability by means of the boundedness of the critical orbits that 
the bifurcation locus in the parameter space $\bC^3$
of the family $(z,w)\mapsto (z^2, w^2 + Az^2 + Bz + C)$
of quadratic polynomial skew products on $\bC^2$ accumulates to the subset
\begin{gather*}
\{[A,B,C]\in\P^2_\infty: Az^2 + Bz + C =0 \text{ for some } z\in S^1
\}
\end{gather*}
in the hyperplane at infinity $\P^2_\infty$ of $\bC^3$, see
\cite[Theorem C]{ab_skew}. 
{Setting}
$\eta=1$,
$A= \frac{H}{tG^2}$, $B=\frac{c}{tG}$, and $C=0$, 
the condition that
$Az^2+Bz+Cz=0$ for some $z\in\bC$
{reduces to}
\[
c + \frac{H}{G}z=0\quad\text{for some } z \in S^1.
\]
This is equivalent to the {\itshape bifurcation 
condition}  $|H/G|=|c|$ that we found in Theorem \ref{teo_2}. 

\subsection*{The topologies of the (second) Julia sets}
In \cite[Theorem D]{ab_skew},
hyperbolic components in the parameter space $\bC^3$
of the family $(z,w)\mapsto (z^2, w^2 + Az^2 + Bz + C)$
of quadratic polynomial skew products on $\bC^2$
near the hyperplane at
infinity $\bP^2_{\infty}$ are also classified;
the proof was based on a monodromy argument, and
distinguished hyperbolic components
in terms of the topologies of the (second) Julia sets.

In our situation $\eta=1$ in \eqref{eq_comparison}, by a similar monodromy argument, 
it is also possible to observe that {for every $G_0\in\bC^*$,}
the {topologies of} the (second) Julia set $J_{f_{t,G,H,1}}$ of $f_{t,G,H,1}$ 
at parameters $|H|\ll 1$ and $|H|\gg {1}$ 
{and $t\in\bD^*$ close enough to $0$} are incompatible,
{so that}
the two kinds of parameters $(G_0,H,t_0)$ for $|H|\ll 1$ and $H \gg {1}$
and $0<|t_0|\ll 1$ {cannot} belong to the same \emph{hyperbolic} component
in the parameter space $\bC$ of the family $(f_{t_0,G_0, H})_{H\in\bC}$
(notice that this is not enough to provide a proof of Theorem \ref{teo_2}).

\begin{proposition}
Let $\mathcal {C}$ be a Cantor set.
Then for every $(G_0,{H_0})\in\bC^*\times\bC$, the following hold$;$
\begin{enumerate}
\item when $|H_0|\ll 1$, for every $t\in\bD^*$ close enough to $0$
{and every $(G,H)$ close enough to $(G_0,H_0)$},  
$J_{f_{t,G,H}}$
is a suspension of $\mathcal {C}$ set over $S^1$.
\item when $|H_0|\gg {1}$, for every $t\in\bD^*$ close enough to $0$
{and every $(G,H)$ close enough to $(G_0,H_0)$},
$J_{f_{t,G,H}}$ is homeomorphic to 
$S^1 \times \mathcal{C}$.
\end{enumerate}
\end{proposition}
The proof is done by an argument
{similar to that}
 in \cite[Section 7]{ab_skew} for the case of polynomial skew products, 
and we would thus omit it.

\section{Proof of Theorem \ref{th:individual}}
\label{sec:individual}

Pick a non-exceptional $(g,h)$ 
and let $f_t=f_t(z,w;g,h)$.
Set $\cC_\eta := \{(x,y)\in\C^2:|y| > \eta |x|>0\}$ for each $\eta>0$. 
Recall that $\|\cdot\|$ denotes the Euclidean norm on $\bC^2$ 
and let $p_2(z,w)=w$ be the projection to the second coordinate.

\begin{claim}\label{lemma_computation_l}
{For every $\delta\in(0,1)$},
there exists $\eta_0>0$ so large that
for every $t\in\bD^*$ close enough to $0$ and every $(z,w) \in J_{f_t}$,
the subset $\cC_{\eta_0}$
is invariant under $(Df_{t})_{(z,w)}$,
that is, $(Df_{t})_{(z,w)}(C_{\eta_0})\subset C_{\eta_0}$, and
for every $(x,y)\in \cC_{\eta_0}$,
\begin{gather}\label{eq:image}
(1-\delta)\cdot |2w|
\le\frac{\bigl\|(Df_{t})_{(z,w)}(x,y)\bigr\|}{\|(x,y)\|}
\le (1+\delta)\cdot |2w|,
\end{gather} 
so in particular that, for every $n\in\bN$, 
\begin{gather}
\bigl(2(1-\delta)\bigr)^n \prod_{j=0}^{n-1}   
\bigl|p_2(f_t^j(z,w))\bigr|
\leq \frac{\bigl\|D(f_{t}^{n})_{(z,w)}(x,y)\bigr\|}{\|(x,y)\|}  
\leq \bigl(2(1+\delta)\bigr)^n \prod_{j=0}^{n-1}
\bigl|p_2(f_t^j(z,w))\bigr|.
\tag{\ref{eq:image}$'$}\label{eq:imageiterate}
\end{gather}
\end{claim}

\begin{proof}
For every $\alpha\in\bC$, every $t\in\bD^*$, and every $(z,w)\in\bC^2$,
we compute as 
\begin{gather*}
 (Df_t)_{(z,w)}\begin{pmatrix} 
	       1 \\
	      \alpha
	      \end{pmatrix}
=
\begin{pmatrix}
2tGz + t{\tilde{g}_z} & 1+tg_w \\
c + th_z & (2w + c_1)+th_w
\end{pmatrix}
\begin{pmatrix}
1\\ \alpha
\end{pmatrix}
=
\begin{pmatrix}
2tGz + t{\tilde{g}_z} + (1+tg_w)\alpha\\
c+th_z + (2w+c_1+th_w) \alpha
\end{pmatrix}.
\end{gather*}
By Lemma \ref{lemma_stime}, for every $\beta\in(0,1)$ small enough,
there exists $\eta_0>0$ so large that 
for every $t\in\bD^*$ close enough to $0$, every $(z,w)\in J_{f_t}$, 
and every $(1,\alpha)\in\cC_{\eta_0}$,  
we have ($J_{f_t}\subset V_t$ and moreover) the
estimates
\begin{align*}
|2tGz+ t{\tilde{g}_z} + (1+tg_w)\alpha|
  \leq&(1+4^{-1}\beta)\cdot(1+|g_{zw}/G|)|\alpha|\quad\text{and}\\
(1-4^{-1}\beta)\cdot 2|w\alpha|\leq |c+th_z + (2w+c_1+th_w)\alpha|
\leq & (1+ 4^{-1}\beta)\cdot 2|w\alpha|.
\end{align*}
{So,} in particular 
\begin{gather*}
 (1-\beta)\cdot \frac{2|w|}{1+|g_{zw}/G|}
 \le
\frac{|c+th_z + (2w+c_1+th_w)\alpha|}{|2tGz+ t{\tilde{g}_z} + (1+tg_w)\alpha|}\quad\text{and}\\
(1-4^{-1}\beta)|2w\alpha| 
  \le \bigl\|(Df_{t})_{(z,w)}(1,\alpha)\bigr\|
\le(1+4^{-1}\beta)\Bigl(\frac{1+|g_{zw}/G|}{|2w|}+1\Bigr)|2w\alpha|.
\end{gather*}
Noting also
that
$\inf_{(z,w)\in V_t}|w|=O(t^{-1/2})$ as $t\to 0$ 
(under the assumption that $(g,h)$ is non-exceptional), 
the invariance $(Df_t)_{(z,w)}(C_{\eta_0})\subset C_{\eta_0}$ holds
and, for every $\delta\in(0,1)$, 
we obtain the desired estimate \eqref{eq:image},
decreasing $\beta\in(0,1)$ and increasing $\eta_0>0$ if necessary.
\end{proof}

Let us first see the final convergence assertion on $\chi_1(f_t),\chi_2(f_t)$
as $t\to 0$. Recall two convergence results from ergodic theory.
For $0<|t|\ll 1$,
by the Oseledec multiplicative ergodic theorem,
for $\mu_{f_t}$-almost every $p\in\bC^2$, the limit
$\Lambda_p:=\lim_{n\to \infty}\bigl((Df_{t}^{n})_p^*(Df_{t}^{n})_p\bigr)^{1/(2n)}$ 
exists in $M(2,\bC^2)$
with respect to the operator norm topology
and
has the two individual Lyapunov exponents $\chi_1(f_t)\ge\chi_2(f_t)$ of $f_t$
as the two {eigenvalues}. Moreover,
{there is a canonical filtration
$\bC^2\supset E_{2,p}\supset\{0\}$ of 
$\bC^2$
by the invariant subspace $E_{2,p}$ under $\Lambda_p$ 
such that}
\[
\lim_{n\to \infty}\frac{1}{n} \log \|D(f_{t}^{n})_p(v)\|
=\begin{cases}
  \chi_1 (f_{t}) & \text{for every }v\in \bC^2\setminus E_{2,p},\\
  \chi_2 (f_{t}) & \text{for every }v\in E_{2,p}\setminus\{0\}.
 \end{cases}
\]
On the other hand, for every $0<|t|\ll 1$, by the Birkhoff 
ergodic theorem, for $\mu_{f_t}$-almost every $(z,w)\in\bC^2$, we also have
\begin{gather*}
 \lim_{n\to \infty}\frac{1}{n}\sum_{j=0}^{n-1}\log
\bigl|p_2(f_t^j(z,w))\bigr|  
 =\int_{\bC^2}\log|w|\mu_{f_t}(z,w).
\end{gather*}
Once 
the Claim and the above two convergence
results
from ergodic theory are at our disposal,
for every $\delta\in(0,1)$, there is $\eta_0>0$ so large that
for every $0<|t|\ll 1$,
we have 
\[
\left|\chi_1 (f_t)-\log 2-\int_{V_t} \log |w|\mu_{f_t}(z,w)
\right|
\le \log(1-\delta)^{-1}.
\]
On the other hand, in the proof of Theorem \ref{teo_1}(ii), we have already seen that
\begin{align*}
&\int_{V_t}\log|w|\mu_{f_t}(z,w)-\frac{1}{2}\log|t|^{-1}\\
& =
\frac{1}{2}\biggl(\int_{V_t}\log|z|\mu_{f_t}(z,w)-\log|t|^{-1}\biggr)
+\frac{1}{2}\int_{V_t}\biggl(\log|c+tHz|+\log\left|1+\frac{o(t^{-1})}{cz+tHz^2}\right|\biggr)\mu_{f_t}(z,w)\\
& =-\frac{1}{2}\log|G|+\frac{1}{2}\log\max\biggl\{|c|,\biggl|\frac{H}{G}\biggr|\biggr\}+o(1)\quad\text{as }t\to 0.
\end{align*}
Hence we have
the convergence
\begin{gather*}
\lim_{t\to 0}\Bigl(\chi_1 (f_t) - \frac{1}{2} \log|t|^{-1}\Bigr)
=\log\frac{2\max\{|c|,|H/G|\}^{1/2}}{|G|^{1/2}}
\end{gather*}
(as $\delta\to 0$),
and in turn 
have
the convergence $\chi_2(f_t)=L(f_t)-\chi_1(f_t)\to
\log 2$
as $t\to 0$,
also by Theorem \ref{teo_1}(ii).

Now let us see the harmonicity assertion on $t\mapsto\chi_1(f_t),\chi_2(f_t)$. 
Recall that by Berteloot--Dupont--Molino \cite[Theorem 1.5]{BDM08}, 
for every $0<|t|\ll 1$,
\begin{gather}
\begin{cases}
 \lim_{n\to\infty}2^{-2n}\sum_{p\in R(f_t^n)\cap J_{f_t}}
 \frac{1}{n}\log\|D(f_t^n)_p\|=\chi_1(f_t)\quad\text{and}\\
 \lim_{n\to\infty}2^{-2n}\sum_{p\in R(f_t^n)\cap J_{f_t}}
 \frac{1}{n}\log|\det(D(f_t^n)_p)|=\chi_1(f_t)+\chi_2(f_t),
\end{cases}\label{eq:BDM}
\end{gather}
where we denote by $R(f_t^n)$ the set of all repelling fixed points of $f_t^n$
in $\bC^2$ and by $\|D(f_t^n)_p\|$ the operator norm of the differential 
$D(f_t^n)_p$ for each $p\in R(f_t^n)\cap J_{f_t}$. 

{We also claim that there is $r_0\in(0,1)$ so small that
for every $t\in\bD_{r_0}^*$, every $n\in\bN$, and every $p\in R(f_t^n)\cap J_{f_t}$, 
the absolute values of the two eigenvalues of the differential $D(f_t^n)_p$ 
are different and $>1$, so in particular 
letting $\lambda_{1,p,n}(t),\lambda_{2,p,n}(t)$ be the two eigenvalues of
$D(f_t^n)_p$ such that $|\lambda_{1,p,n}(t)|>|\lambda_{2,p,n}(t)|$,
the above approximations in \eqref{eq:BDM} yield
\begin{align*}
\begin{cases}
\lim_{n\to\infty}2^{-2n}\sum_{p\in R(f_t^n)\cap J_{f_t}}\frac{1}{n}\log|\lambda_{1,p,n}(t)|=\chi_1(f_t)\quad\text{and}\\
\lim_{n\to\infty}2^{-2n}\sum_{p\in R(f_t^n)\cap J_{f_t}}\frac{1}{n}\log|\lambda_{2,p,n}(t)|=\chi_2(f_t);
\end{cases}
\tag{\ref{eq:BDM}$'$}\label{eq:BDMeigen}
\end{align*}
}indeed, by the above Claim, fixing $\delta\in(0,1)$ and then
$\eta_0>0$ large enough, for every $0<|t|\ll 1$ 
and every $p=(z,w)\in R(f_t^n)\cap J_{f_t}$, 
the M\"obius transformation $A$ induced by $D(f_t^n)_p\in\mathrm{GL}(2,\bC)$
on the projectivization $\P^1$ of 
$\bC^2\setminus\{(0,0)\}$
maps the open spherical disk $\mathcal{D}_{\eta_0}$ in $\bP^1$ 
corresponding to the cone $\cC_{\eta_0}\cup\{(x,y):x=0\}$ 
minus $\{(0,0)\}$ 
to a relatively compact subset in $\mathcal{D}_{\eta_0}$.
This implies the existence of a fixed point of 
$A$ in $\mathcal{D}_{\eta_0}$, and in turn that of 
an eigenvector $v_1$ of $D(f_t^n)_p$ in $\cC_{\eta_0}$. 
Setting $(z_j,w_j):=f_t^{j-1}(p)$ for each $j\in\{1,\ldots,n\}$,
the eigenvalue $\lambda_1$ of $D(f_t^n)_p$ associated to $v_1$ 
satisfies 
$2(1+\delta)|w_1\cdots w_n|^{1/n}\ge|\lambda_1|^{1/n}
\geq 2(1-\delta)|w_1\cdots w_n|^{1/n}\to\infty$
as $t\to 0$, 
and the other eigenvalue $\lambda_2$ of $D(f_t^n)_p$ then satisfies
\begin{multline*}
\frac{2+o(1)}{1+\delta}=\biggl(\prod_{j=1}^n\frac{|4tGz_jw_j+O(1)|}{2(1+\delta)\cdot|w_j|}\biggr)^{1/n}
\le\\
\le|\lambda_2|^{1/n}=\biggl(\frac{|\det D(f_t^n)_p|}{|\lambda_1|}\biggr)^{1/n}
\le\biggl(\prod_{j=1}^n\frac{|4tGz_jw_j+O(1)|}{2(1-\delta)\cdot|w_j|}\biggr)^{1/n}
=\frac{2+o(1)}{1-\delta}\quad\text{as }t\to 0,
\end{multline*}
and both the divergence of $|\lambda_1|^{1/n}$ and 
the bounds of $|\lambda_2|^{1/n}$ as $t\to 0$ are uniform on $n,p$.
Hence the claim holds. 

Recall now that
 the function $t\mapsto L(f_t)$ is harmonic on $0<|t|\ll 1$
(seen in Theorem \ref{teo_1}(ii)),
so that by \cite[(B)$\Rightarrow$(A) in Theorem 1.1]{bbd},
for any simply connected subdomain $D$ in $0<|t|\ll 1$, 
the set function $t\mapsto\bigcup_{n\in\bN}R(f_t^n)\cap J_{f_t}$ is regarded as
a holomorphic motion parametrized by $D$.
{Consequently, applying (a variant of) Harnack's theorem 
(see e.g.\ \cite[Theorem 1.3.10]{Ransford95}) 
to the {\itshape convergent} sequences
of {\itshape positive harmonic} functions on $D$ in \eqref{eq:BDMeigen},
both the functions 
$t\mapsto\chi_1(f_t)$ and $t\mapsto\chi_2(f_t)$ are harmonic on $D$.}
\qed

\begin{acknowledgement}
This 
{project}
grew up during the second author's stay
at Imperial College London 
and the first author's stay
at RIMS and Mathematics Department, Kyoto University 
in Autumn 2017, and the authors thank those institutes for their hospitality.
The first author's visit was partially supported by 
JSPS Grant-in-Aid for Scientific Research (B) 22340025.
The second author was partially supported by JSPS Grant-in-Aid 
for Scientific Research (C), 15K04924. 
\end{acknowledgement}

\bibliographystyle{alpha}
\bibliography{biblio_bo}

\end{document}